\documentclass[letterpaper,reqno, 12pt]{amsart}

\usepackage[top=1in, bottom=1in, left=1in, right=1in]{geometry}
\usepackage{amsmath}
\usepackage{amsthm}
\usepackage{amsfonts}
\usepackage{amssymb}
\usepackage{mathrsfs}
\usepackage{url}
\usepackage{enumerate}
\usepackage[shortlabels]{enumitem}
\usepackage{multirow}
\usepackage{float}
\usepackage{mathtools}
\usepackage{tikz-cd}
\usepackage{bm}
\usepackage{hyperref}
\hypersetup{pdfauthor={Tim Davis},%
	            pdftitle={Fourier Coefficients of Hilbert Modular Forms at Cusps},%
	            pdfsubject={Number Theory},%
	            pdfcreator={XeLaTeX},%
	            colorlinks=true,%
	            linkcolor=[rgb]{0,0,1},
	            urlcolor = blue,
	            citecolor=red}
\definecolor{cites}{rgb}{0.75 , 0.00 , 0.00}  
\definecolor{urls} {rgb}{0.00 , 0.00 , 1.00}  
\definecolor{links}{rgb}{0.00 , 0.00 , 0.5}   

\def\Z{{\mathbb Z}}

\def\R{{\mathbb R}}

\def\C{{\mathbb C}}
\def\H{{\mathbb H}}
\def\Q{{\mathbb Q}}
\def\A{{\mathbb A}}
\def\O{{\mathcal O}}
\def\a{{\mathfrak a}}
\def\p{{\mathfrak p}}
\def\n{{\mathfrak n}}
\def\D{{\mathfrak D}}
\def\Qp{{\mathbb{Q}_p}}
\def\AF{{\mathbb{A}_F}}

\def\iotaf{{\iota_\text{f}(\sigma^{-1})}}

\def\congsub{{\Gamma_\mu(\n)}}
\def\E(Q){{E(\mathbb{Q})}}

\def\bE(Q){{\bar{E}(\mathbb{Q})}}
\def\eps{\varepsilon}

\def\gl2R{\GL_2(\mathbb{R})}
\def\sl2Z{\SL_2(\mathbb{Z})}
\def\af{a_f(n;\sigma)}

\def\smatrix{(\begin{smallmatrix} a & b \\ c & d \end{smallmatrix})}

\def\bs{\backslash}
\def\f{{\mathbf f}}
\DeclareMathOperator{\Ima}{Im}

\DeclareMathOperator{\GL}{GL}
\DeclareMathOperator{\SL}{SL}

\DeclareMathOperator{\Aut}{Aut}
\DeclareMathOperator{\Gal}{Gal}
\DeclareMathOperator{\Tr}{Tr}

\newtheorem{IntroThm}{Theorem}
\newtheorem{Thm}{Theorem}[section]
\newtheorem{Lemma}[Thm]{Lemma}

\newtheorem{Prop}[Thm]{Proposition}

\theoremstyle{definition}

\newtheorem{Question}[Thm]{Question}

\theoremstyle{remark}
\newtheorem{Remark}[Thm]{Remark}

\begin{document} 

\title[Fourier coefficients at cusps]{Fourier Coefficients of Hilbert Modular Forms at Cusps}
\author{Tim Davis}

\begin{abstract}
The aim of this article is to study the fields generated by the Fourier coefficients of Hilbert newforms at arbitrary cusps. Precisely, given a cuspidal Hilbert newform $f$ and a matrix $\sigma$ in (a suitable conjugate of) the Hilbert modular group, we give a cyclotomic extension of the field generated by the Fourier coefficients at infinity which contains all the Fourier coefficients of $f||_k\sigma$. 
\end{abstract}

\maketitle

\tableofcontents

\section{Introduction}

Let $f$ be a normalised Hecke eigenform for $\Gamma_0(N)$ of weight $k$. It is known that the field generated by the Fourier coefficients of $f$ is a number field, see \cite[Proposition 2.8]{Shi78}. This number field we denote by $\Q(f)$. Furthermore, for any matrix $g\in\SL_2(\Z)$, an application of the $q$-expansion principle shows that the Fourier coefficients of $f|_kg$ lie in the cyclotomic extension $\Q(f)(\zeta_N)$, see \cite[Remark 12.3.5]{MFandMC} for more details. In recent years, algorithms have been developed to compute these Fourier coefficients. These were developed by several different authors, see \cite{HC,MDMN} for more details about these algorithms. Their methods use the knowledge that these numbers are in fact algebraic numbers. This means if we know more about the number field where these Fourier coefficients lie, this could speed up the computations. \\ 

Therefore the question is given a Hecke eigenform (which we may assume to be a newform) $f$ of level $N$ and weight $k$ and $\sigma\in\SL_2(\Z)$; what is the number field that the Fourier coefficients of $f|_k\sigma$ generate? Or a slightly weaker question is: can one write down an explicit subfield of $\Q(f)(\zeta_N)$, depending on the entries of $\sigma$, which contains all the Fourier coefficients of $f|_k\sigma$? This was answered in a paper of Brunault and Neururer, who proved the following result \cite[Theorem 4.1]{FEatC}.  \\

\textit{Let $f$ be a normalised newform on $\Gamma_0(N)$ with weight $k$. Let $\Q(f)$ be the field generated by all the Fourier coefficients of $f$. Let $\sigma=(\begin{smallmatrix} a & b \\ c & d \end{smallmatrix})\in\sl2Z$. Then the Fourier coefficients of $f|_k \sigma$ lie in the cyclotomic extension $\Q(f)(\zeta_{N'})$ where $N'=N/(cd,N)$}.\\

The method of Brunault and Neururer was classical. They use a result of Shimura \cite[Theorem 8]{Shi75}, which studied the connections between two actions on spaces of modular forms: the action of $\GL_2^+(\Q)$ via the slash-operator and the action of $\Aut(\C)$ on the Fourier coefficient of a modular form. The proof of Brunault and Neururer also applies to modular forms of $\Gamma_0(N)$ that are not necessarily newforms but in that case it is not known if the field is optimal. In this paper we give a new proof of the result of Brunault and Neuruer as well as a substantial generalisation (to the case of Hilbert modular forms) using adelic and local representation-theoretic methods. Specifically we use local Whittaker functions and their invariance properties. The starting point is an explicit formula for the Fourier coefficients of $f|_k\sigma$ given in \cite[Section 3]{ACAS}, in terms of local Whittaker functions. From there we find sufficient conditions for $\tau\in\Aut(\C)$ to fix this product. The advantage of this method is that it gives an insight into how one could prove analogous results for other families of modular forms, specifically it can be generalised to Hilbert modular forms. One would suspect that a similar method can also be used to prove a result of the same style for Whittaker coefficients (or other factorisable periods) of automorphic forms lying in cohomological automorphic representations of higher rank groups.\\ 

The main result of this paper is an extension of Brunault and Neururer stated above to the case of cuspidal Hilbert newforms. Let $F$ be a totally real number field of degree $n$ with narrow class group of size $h$. Let $\n$ be a fixed integral ideal of $\O_F$. For $\mu=1,...,h$, we define the congruence subgroup $\Gamma_\mu(\n)$ of $\GL_2(F)$ as \[\congsub=\bigg\{\begin{pmatrix} a & b \\  c & d \end{pmatrix}: a,d \in\O_F, b\in(t_\mu)^{-1}\mathfrak{D}_F^{-1}, c\in\n t_\mu\mathfrak{D}_F,ad-bc\in\O_F^\times\bigg\},\] where $\D_F$ is the absolute different of $F$ and $t_\mu$ is a representative of the narrow class group. \\ 

We identify $\alpha\in\GL_2(F)$ with an element of $\GL_2(\R)^n$ via the various embeddings of $F$ in $\R$ and for $y=(y_1,...,y_n)\in\C^n$ and $k=(k_1,...,k_n)\in\Z^n$ we use the notation $y^k:=y_1^{k_1}...y_n^{k_n}$.\\  

Then $f:\H^n\rightarrow\C$ is a Hilbert modular form of weight $k=(k_1,...,k_n)$ and level $\congsub$ if $f$ is holomorphic on $\H^n$ and cusps and we have $f||_k\alpha(z):=\det\alpha^{k/2}(cz+d)^{-k}f(\alpha z)=f(z)$ for every $\alpha=\smatrix\in\Gamma_\mu(\n)$. Any $f$ of the above form has a Fourier expansion of the form \[f(z)=\sum_\xi a(\xi;f)e^{2\pi i\Tr(\xi z)}.\] Following Shimura \cite[2.24]{Shi78}, we define for $f$ as above, $1\leq\mu\leq h$, \[c_\mu(\xi;f)=N(t_\mu\O_F)^{-k_0/2}a(\xi;f)\xi^{(k_0\mathbf{1}-k)/2},\] where $\mathbf{1}=(1,...,1)$ and $k_0=\max\{k_1,...,k_n\}$.\\

We say $f$ is a Hilbert cuspform if the the constant term in the Fourier expansion of $f||_k\gamma$ is zero for every $\gamma\in\GL_2^+(F)$. A cuspidal Hilbert newform of weight $k$ and level $\n$ is a tuple $\f=(f_1,...,f_h)$ where $f_\mu$ is a Hilbert cuspform for $\Gamma_\mu(\n)$ and such that $\f$ satisfies additional properties (it does not come from a form of lower level and is a Hecke eigenform), see \cite[Section 2]{Shi78} for more details. A normalised cuspidal Hilbert newform $\f$ (also called primitive form) is one such that the first Fourier coefficient of $\f$ (which is non-zero) is normalised to equal 1. A result of Shimura \cite[Proposition 2.8]{Shi78}, tells us that the set of all $\{c_\mu(\xi;f_\mu):1\leq\mu\leq h \ \text{and} \ \xi\in F_+\}$ generates a totally real or CM number field, denoted $\Q(\f)$, under the assumption that $k_1\equiv...\equiv k_n \pmod{2}$. We prove the following result.

\begin{IntroThm}[Theorem \ref{Thm}]\label{HMFThm}
Let $\f=(f_1,...,f_h)$ be a normalised Hilbert newform of weight $k=(k_1,...,k_n)$ with $k_1\equiv...\equiv k_n \pmod{2}$ and level $\n$. Let $1\leq\mu\leq h$ and let $\sigma=\smatrix\in\Gamma_\mu(1)$. Let $f_\mu||_k\sigma$ have the Fourier expansion \[f_\mu||_k\sigma(z)=\sum_{\xi}a_\mu(\xi;\sigma)e^{2\pi i \Tr(\xi z)},\] and define \[c_\mu(\xi;f_\mu||_k\sigma)=N(t_\mu\O_F)^{-k_0/2}a_\mu(\xi;\sigma)\xi^{(k_0\mathbf{1}-k)/2}.\] Then $c_\mu(\xi;f_\mu||_k\sigma)$ lie in the number field $\Q(\f)(\zeta_{N_0})$ where $N_0$ is the integer such that $N_0\Z=\n/(cdt_\mu^{-1}\mathfrak{D}_F^{-1},\n)\cap\Z$.
\end{IntroThm} 

To prove this theorem we study the classical action of $\Aut(\C)$ on the Fourier coefficients via the action of $\Aut(\C)$ on local Whittaker newforms. We give an explicit formula for the Fourier coefficients of a classical Hilbert automorphic form in terms of a global Whittaker function. We can then apply this general result to the specific newform in the theorem. This generalises two results of \cite{ACAS}, namely Lemma 3.1 and Propostion 3.2, where the authors proved analogous results for modular forms. The global Whittaker newform can be broken up into a product of local Whittaker newforms. Therefore we have that we can write the Fourier coefficients of $f_\mu||_k\sigma$ in terms of local Whittaker newforms. Thus studying the action of $\Aut(\C)$ on these Fourier coefficients is equivalent to studying the action of $\Aut(\C)$ on local Whittaker newforms.

\subsection*{Acknowledgments}

The author would like to thank his supervisor, Abhishek Saha, for suggesting this problem and for many helpful discussions. This research is supported by the Leverhulme Trust, grant number RPG-2018-401.

\section{Background Setting and Notation}

Here we discuss the background setting and collect the notation which will be used throughout this paper.

\subsection{Background notation}\label{action}

For an integer $N\geq1$, we define $\zeta_N=e^{2\pi i/N}$. For ease of notation we will write $e^{2\pi i x}$ as $e(x)$. We will use the shorthand for the following matrices, $a(x):=\big(\begin{smallmatrix} x & \\ & 1\end{smallmatrix}\big)$ and $n(x):=\big(\begin{smallmatrix} 1 & x\\ & 1\end{smallmatrix}\big)$. We denote the upper half plane by $\H$. We let $\GL_2^+(\R)$ denote the real two by two matrices with positive determinant. This group acts on $\H$ via the action $\smatrix z=\frac{az+b}{cz+d}$. For a representation $\Pi$ we define the representation ${}^{\tau}\Pi$ as follows. Let $V$ be the space of $\Pi$ and let $V'$ be any vector space such that $t:V\rightarrow V'$ is a $\tau$-linear isomorphism. We define the representation $({}^{\tau}\Pi,V')$ via ${}^{\tau}\Pi(g)=t\circ\Pi(g)\circ t^{-1}$.

\subsection{Number fields}

Let $F$ be a totally real number field of degree $n$, with narrow class group of size $h$. We denote the ring of integers of $F$ by $\O_F$. The embeddings of $F$ in $\C$ we write as $\eta_1,...,\eta_n$. With respect to these embeddings we have that $F$ naturally embeds into $\R^n$. For an element $\alpha$ of $F$ write $(\alpha_1,...,\alpha_n)$ for $(\eta_1(\alpha),...,\eta_n(\alpha))$ as an element of $\R^n$. We write $\alpha^k=\prod_{j=1}^n\alpha_j^{k_j}$ for $k=(k_1,...,k_n)$. For a subset $S$ of $F$, $S_+$ denotes the totally positive elements of $S$. Throughout $\n$ will be a fixed integral ideal of $F$. We denote the trace map of $F$ to $\Q$ by $\Tr$. Let $\D_F$ denote the absolute different of $F$, that is, $\D_F^{-1}=\{x\in F: \Tr(x\O_F)\subset\Z\}$. Let $F_v$ be the completion of $F$ at a non-archimedean place $v$. The ring of integers of $F_v$ are denoted by $\O_v$. We denote the maximal ideal of $\O_v$ by $\p_v$ and a generator of $\p_v$ by $\varpi_v$. We write $\n_v=\n\otimes_{\O_F}\O_v$ and $\mathfrak{D}_v=\mathfrak{D}_F\otimes_{\O_F}\O_v$ to be the $v$-part of $\n$ and $\D_F$ respectively. For an ideal $\n_v$ of $\O_v$ we define the subgroup $K_v(\n_v)$ of $\GL_2(F_v)$ as
\begin{equation}\label{Kv}
K_v(\n_v):=\bigg\{\begin{pmatrix} a & b\\ c& d \end{pmatrix}\in\GL_2(F_v): \begin{aligned} & a\O_v+\n_v=\O_v, \ b\in\mathfrak{D}_v^{-1} \\ & c\in\n_v\mathfrak{D}_v, \ d\in\O_v \end{aligned}, ad-bc\in\O_F^\times\bigg\}.
\end{equation}
By abuse of notation, given an ideal $\n$ of $\O_F$ use $K_v(\n)$ to denote $K_v(\n_v)$ where $\n_v=\n\otimes_{\O_F}\O_v$ is the $v$-part of $\n$. We define $n_v$ and $d_v$ via $\n_v=\varpi_v^{n_v}\O_v$ and $\D_v=\varpi_v^{d_v}\O_v$ respectively. 

\subsection{Adeles}

The ring of adeles of $F$ is denoted by $\A_F$. In the specific case that $F=\Q$ we denote $\A_\Q=\A$. The narrow class group can be viewed as \[F^\times\bs\A_F^\times/F_{\infty+}^\times\prod\O_v^\times.\] We fix elements, for $1\leq\mu\leq h$, $t_\mu\in\A_F^\times$ such that \[
(t_\mu)_v
\begin{cases}
=1, \ \text{if}\  v\in\infty \ \text{or} \  v | \n \\
\in\O_v, \ \text{if} \ v\nmid\infty \  \text{and} \ v\nmid\n
\end{cases}
\] and $\{t_\mu\}$ are representatives of the narrow class group. We let $(t_\mu\O_F)$ denote the ideal of $\O_F$ corresponding to $t_\mu$. We have the following disjoint union decomposition \[\A_F^\times=\bigcup_{1\leq\mu\leq h} t_\mu F^\times F^\times_{\infty,+}\prod_{v<\infty}\O_v^\times.\] We denote $x_\mu=(\begin{smallmatrix}1 & \\ & t_\mu\end{smallmatrix})$ and for a finite place $v$ we denote $x_{\mu,v}=(\begin{smallmatrix}1 & \\ & (t_\mu)_v\end{smallmatrix})$. For $g\in\GL_2(\A_F)$ we define $\iota_\text{f}(g)$ to equal $g$ at all finite places and the identity at infinite places.

\subsection{Measures}

We normalise the measures so that Vol$(\O_v)=1$ and in a multiplicative setting Vol$(\O_v^\times)=1$. We also normalise so that Vol$(F\bs\AF)=1$.

\subsection{Additive characters}

Let $\psi_\Q:\Q\bs\A\rightarrow\C^\times$ be an additive character defined by $\psi_\Q=\prod_p\psi_{\Q,p}$ where $\psi_{\Q,\infty}(x)=e(x)$ for $x\in\R$ and $\psi_{\Q,p}(x)=1$ for $x\in\Z_p$. We then define an additive character $\psi$ on $F\bs\A_F$ by composing $\psi_\Q$ with the trace map from $F$ to $\Q$, that is, $\psi=\psi_\Q\circ\Tr:F\bs\A_F\rightarrow\C^\times$. If $\psi=\otimes\psi_v$ then the local characters are defined analogously and we have $\psi_\infty(x)=e(\Tr(x))$. Using \cite[Lemma 2.3.1]{RS} we have a formula for computing the value of the exponent of the character $\psi_v$ where $v$ is a non-archimedean place of $F$. For a place $v$ of $F$, letting $c(\psi_v)$ denote the smallest integer $c_v$ such that $\psi_v$ is trivial on $\p_v^{c_v}$, we have by \cite[Lemma 2.3.1]{RS}
\begin{equation}\label{formula}
c(\psi_v)=d_v,
\end{equation} 
where $d_v$ is defined as above, that is, via $\D_v=\varpi_v^{d_v}\O_v$.

\section{The Whittaker Model}

Recall that each infinite dimensional admissible representation of $\GL_2$ over a local field admits a unique Whittaker model see \cite[Chapter 3]{Bump} and \cite[Chapter 4]{DGJH}. Furthermore each cuspidal automorphic representation of $\GL_2$ admits a unique global Whittaker model. We give below the definition and basic properties of the Whittaker model for representation of trivial central character.\\

For any place $v$ of $F$ let $\Pi_v$ be an irreducible admissible infinite dimensional representation of $\GL_2(F_v)$. Let $\mathcal{W}(\psi_v)$ be the space of smooth functions $W_v:\GL_2(F_v)\rightarrow\C$ satisfying \[W_v\Bigg(\begin{pmatrix} 1 & x \\ & 1 \end{pmatrix} g\Bigg)=\psi_v(x)W_v(g), \ \text{for} \ x\in F_v \ \text{and} \ g\in\GL_2(F_v).\] The local Whittaker model, denoted $\mathcal{W}(\Pi_v,\psi_v)$, is the unique subspace of $\mathcal{W}(\psi_v)$ such that the representation of $\GL_2(F_v)$ on the space $\mathcal{W}(\Pi_v,\psi_v)$ (given by right translation) is equivalent to the representation $\Pi_v$.\\

Let $(\Pi,V_\Pi)$ be a cuspidal automorphic representation of $\GL_2(\A_F)$. Then the global Whittaker model, $\mathcal{W}(\Pi,\psi)$, for $\Pi$ with respect to a fixed non-trivial additive character $\psi$, consists of the space generated by functions $W_\phi$ on $\GL_2(\AF)$ given by 
\begin{equation}\label{Def}
W_\phi(g):=\int_{\A_F/F}\phi\Bigg(\begin{pmatrix} 1 & x \\ & 1 \end{pmatrix}g\Bigg)\overline{\psi(x)}\ dx,
\end{equation}
as $\phi$ varies in $V_\Pi$. The representation $\mathcal{W}(\Pi,\psi)$ decomposes as a restricted tensor product of local Whittaker models, $\mathcal{W}(\Pi_v,\psi_v)$.\\

For $\Pi_v$ an irreducible admissible infinite dimensional representation of $\GL_2(F_v)$, with unramified central character, let $a(\Pi_v)$ be the smallest integer $n$ such that $\Pi_v$ has a $K_v(\p_v^n)$-fixed vector. Then the normalised Whittaker newform (with respect to $\psi_v$) is the unique $K_v(\p_v^{a(\Pi_v)})$-invariant vector $W_v\in\mathcal{W}(\Pi_v,\psi_v)$ satisfying $W_v(1)=1$.\\

We now state a key result which will be used throughout this paper.

\begin{Lemma}\label{key}
Let $\Pi_v$ be an irreducible admissible infinite dimensional representation of $\GL_2(F_v)$ with unramified central character. Let $K_v(\p_v^{a(\Pi_v)})$ be defined as in \eqref{Kv}. Let $W_v$ be the local normalised Whittaker newform. Let $\tau\in\Aut(\C)$ and $g\in\GL_2(F_v)$. Then we have $\tau(W_v(g))=W_{{}^{\tau}\Pi_v}(a(\alpha_\tau)g)$ where $\alpha_\tau$ is defined via 
\begin{align*}
\Aut(&\C/\Q) &&\rightarrow && \Gal(\overline{\Q}/\Q) &&\rightarrow &&  \Gal(\Q(\mu_\infty)/\Q)&&\rightarrow  &&\hat{\Z}^\times\cong\Pi_v\Z_v^\times &&\subset && \Pi_v\Pi_{\p|v}\O_\p^\times\\
&\tau &&\mapsto &&  \qquad\tau|_{\overline{\Q}} && \mapsto  &&\qquad\tau|_{\Q(\mu_\infty)} && \mapsto  &&\qquad\alpha_\tau &&\mapsto &&\alpha_\tau=(\alpha_{\tau,\p})_\p.
\end{align*}
\end{Lemma}
\begin{proof}
See \cite[Section 2]{ACAS} and \cite[Section 3.2.3]{HMF}.
\end{proof}

Let $\Pi_v$ be an irreducible admissible infinite dimensional representation of $\GL_2(F_v)$ with unramified central character and normalised Whittaker newform $W_v$. Then consider the representation $\Pi_v\otimes| \ |_v^{k_0/2}$, for some integer $k_0$. We now apply the above result to this representation. Therefore using Lemma \ref{key} we have \[\tau\big(W_{\Pi_v\otimes| \ |^{k_0/2}}(g)\big)=W_{{}^{\tau}(\Pi_v\otimes| \ |^{k_0/2})}(a(\alpha_\tau)g).\] This is equivalent to \[\tau\big(W_v(g)\big)\tau\big(|\det(g)|^{k_0/2}_v\big)=W_{{}^{\tau}(\Pi_v\otimes| \ |^{k_0/2})}(a(\alpha_\tau)g).\] So if $\Pi_v\otimes| \ |^{k_0/2}\cong{}^{\tau}(\Pi_v\otimes| \ |^{k_0/2})$, we have that 
\begin{equation}\label{action}
\tau\big(W_v(g)\big)\tau\big(|\det(g)|^{k_0/2}_v\big)=W_v(a(\alpha_\tau)g)|\det(g)|^{k_0/2}_v.
\end{equation}

\section{Modular Forms Case}\label{MFcase}

In this section we study the case of modular forms. The reason we do this is because the proof of the theorem stated below follows a similar method to the one used to prove Theorem \ref{HMFThm} but is less technical. As a consequence of this we are able to prove Thereon \ref{mf2} below, which is an alternative version of \cite[Theorem 4.1]{FEatC} whereby we make some simplifying assumptions on the matrix $\sigma\in\SL_2(\Z)$. The general form of \cite[Theorem 4.1]{FEatC} will follow from applying Theorem \ref{HMFThm} in the setting of classical newforms. 

\subsection{Fourier expansion}

For a function $f:\H\rightarrow\C$ we define a function  $f|_kg$ on $\H$ for an integer $k$ and some $g=\smatrix\in\GL_2^+(\R)$ as $f|_kg(z)=(\det g)^{k/2}(cz+d)^{-k}f(gz)$. Let $N\geq1$. Let $\Gamma$ be a congruence subgroup of $\SL_2(\Z)$. Let $f$ be a modular form on $\Gamma$ of weight $k$. Then we have that $f$ has a Fourier expansion at infinity of the form \[f(z)=\sum_{n\geq0}a_f(n)e^{2\pi i nz/w},\] where $a_f(n)$ are the Fourier coefficients of $f$ and $w$ is the smallest integer such that $(\begin{smallmatrix} 1 & w \\ 0 & 1 \end{smallmatrix})\in\Gamma$. Note that in the case that $\Gamma=\Gamma_0(N)$ we have that $w=1$. If the first Fourier coefficient of $f|_kg$ is zero for every $g\in\SL_2(\Z)$ then we say that $f$ is a cuspform. If we also have that $a_f(1)=1$ we say that $f$ is normalised. Recall that if $f$ is a modular form for $\Gamma_0(N)$ which is non-zero we have that $k$ is even.\\

Let $\sigma\in\SL_2(\Z)$ be such that $\a=\sigma\infty$, where $\a$ is a cusp of $\Gamma_0(N)\bs\H$. If $f$ is a modular form for $\Gamma_0(N)$ then $f|_k\sigma$ is a modular form for the congruence subgroup $\sigma^{-1}\Gamma_0(N)\sigma$. We have that $f|_k\sigma$ has a Fourier expansion of the form \[f|_k\sigma(z)=\sum_{n>0}a_f(n;\sigma)e^{2\pi inz/w(\a)},\] where $a_f(n;\sigma)$ are the Fourier coefficients of $f|_k\sigma$ and $w(\a)$ is the width of the cusp. The width of the cusp is defined to be the integer $N/(L^2,N)$ where $L$ is the denominator of the cusp $\a$.

\subsection{Result for modular forms}

\begin{Thm}\label{mf2}
Let $f$ be a normalised newform on $\Gamma_0(N)$ with weight $k$. Let $\Q(f)$ be the field generated by all the Fourier coefficients of $f$. Let $\a=\sigma\infty$ be a cusp of $\Gamma_0(N)\bs\H$ with $\sigma=(\begin{smallmatrix} a & b \\ L & d \end{smallmatrix})\in\sl2Z, \ L|N, \ (a,N)=(d,N)=1$, so that $\a$ corresponds to the cusp $a/L$. Then the Fourier coefficients of $f|_k \sigma$ lie in the cyclotomic extension $\Q(f)(\zeta_{N/L})$.
\end{Thm}

\begin{Remark}
Recall that to show a complex number $\rho\in\C$ is an element of a number field it is equivalent to showing that $\rho$ is fixed by the Galois group of that number field.
\end{Remark}

\begin{proof}
Let $\tau\in\Aut(\C)$ fix $\Q(f)(\zeta_{N/L})$. So in particular $\tau$ fixes $\Q(f)$ and all the $\zeta_{N/L}$ roots of unity. \\

Let $n=n_0\prod_{p|N}p^{n_p}$ with $(n_0,N)=1$. In the proof of Proposition 3.3 in \cite{ACAS} the authors give an explicit formula for $a_f(n,\sigma)$, given by \[a_f(n,\sigma)=\frac{a_f(n_0)}{n^{k/2}}\Big(\frac{n}{\delta(\a)}\Big)^{k/2}\prod_{p|N}W_p (a(n/\delta(\a))\sigma^{-1}),\] where $\delta(\a)$ is a rational number related to the cusp $\a$ and $W_p$ are the local Whittaker newforms associated to $f$. Note that $\frac{n}{\delta(\a)}\in\Q$ and since $k$ is even we have that the second term is automatically fixed by $\tau$. Since $\tau$ fixes $\Q(f)$ we have that $\tau$ fixes $a_f(n_0)$. \\

Since $\tau$ fixes $\zeta_{N/L}$ we have that $\alpha_\tau\equiv1 \pmod{p^{n_p-l_p}}$. This implies \[\sigma\begin{pmatrix} \alpha_\tau & 0 \\ 0 & 1\end{pmatrix}\sigma^{-1}=\begin{pmatrix}ad\alpha_\tau-bc & ab(1-\alpha_\tau) \\ dL(\alpha_\tau-1) & ad-bL\alpha_\tau\end{pmatrix}\in K_p(p^{n_p}).\] Since $\tau$ fixes $\Q(f)$ we have that $\pi_p={}^{\tau}\pi_p$. Therefore using Lemma \ref{key} we have that \[\tau(W_p(a(n/\delta(\a))\sigma^{-1}))=W_p(a(\alpha_\tau)a(n/\delta(\a))\sigma^{-1}).\] Recall that $W_p$ is right invariant by $K_p(p^{n_p})$. Thus from the condition that $\sigma a(\alpha_\tau)\sigma^{-1}\in K_p(p^{n_p})$ we have that  \[W_p(a(n/\delta(\a))\sigma^{-1})=W_p(a(\alpha_\tau)a(n/\delta(\a))\sigma^{-1}).\] Thus $\tau$ fixes $\af$ and so $\af\in\Q(f)(\zeta_{N/L})$.
\end{proof}

\section{Preliminary Results}\label{Prelim}

In this section we define the classical Hilbert newform and how we can reinterpret this definition in the adelic setting. We then prove some general results needed to prove Theorem \ref{HMFThm}. Specifically we show that the Fourier coefficients of Hilbert newform equals the global Whittaker newform. 

\subsection{Hilbert Modular Forms}\label{Setup}

We define \[\H^n=\{z=(z_1,...,z_n):\Ima(z_j)>0, \forall j=1,...,n\},\] where we write $z=x+iy$ to mean for $z=(z_1,...,z_n)$ and $z_j=x_j+iy_j$ for each $j$. We have that $\GL_2(F)$ acts on $\H^n$ in the following way \[g\cdot z=\Bigg(\frac{\eta_1(a_1)z_1+\eta_1(b_1)}{\eta_1(c_1)z_1+\eta_1(d_1)},...,\frac{\eta_n(a_n)z_n+\eta_n(b_n)}{\eta_n(c_n)z_n+\eta_n(d_n)}\Bigg).\] We define the congruence subgroup $\congsub$ of $\GL_2(F)$ for each $\mu$ as 
\begin{equation}\label{congsub}
\congsub=x_\mu(\GL_2^+(F_\infty)K_0(\n))x_\mu^{-1}\cap\GL_2(F),
\end{equation}
where 
\begin{equation}\label{K0}
K_0(\n)=\prod_{v<\infty}K_v(\n).
\end{equation}

Then $f_\mu:\H^n\rightarrow\C$ is a Hilbert modular form of level $\congsub$ and weight $k=(k_1,...,k_n)$ if $f_\mu$ is holomorphic on $\H^n$ and at cusps and \[f_\mu||_k\alpha(z):=\det\alpha^{k/2}(cz+d)^{-k}f_\mu(\alpha\cdot z)=f_\mu(z),\] for every $\alpha\in\congsub$. 

\begin{Remark}
One can also consider this definition with a character $\chi$ on $\A_F^\times/F^\times$, which we will assume is trivial. For a treatment of $\chi$ being non-trivial see \cite[Section 4]{HMF} and the references therein.  
\end{Remark}

If the constant term in the Fourier expansion of $f_\mu||_k\gamma$ is zero for every $\gamma\in\GL_2^+(F)$ then we say $f_\mu$ is a cuspform. For the space of cuspforms to be nonempty we require that $k_j\geq1$, see \cite[Section 1.7]{HolHMF} for a proof of this fact. In this case $f_\mu$ has a Fourier expansion at infinity of the form \[f_\mu(z)=\sum_{\xi\in(t_\mu\O_F)_+}a_\mu(\xi)e^{2\pi i \text{Tr}(\xi z)}.\] The space of cuspforms of level $\congsub$ and weight $k$ is denoted $S_k(\congsub)$. For each $\mu$ choose a cuspform in $S_k(\congsub)$ and put $\f=(f_1,...,f_h)$.  Then we say that $\f$ is a cuspform of level $\n$ and weight $k=(k_1,...,k_n)$. The space of all such $\f$ is denoted by $S_k(\n)$ and we have \[S_k(\n)=\bigoplus_{\mu=1}^h S_k(\congsub).\] An element $\f\in S_k(\n)$ is said to be a cuspidal classical Hilbert newform if it lies in the orthogonal complement of the oldspace and is a common eigenfunction for the Hecke operators at almost all primes, see \cite[Section 2]{Shi78}.

\subsection{Classical and adelic correspondence}\label{dictionary}

Here we give a description on how one can go between classical Hilbert newforms and adelic newforms (as in Section \ref{Setup} we only consider the case of trivial character $\chi$). We will give an outline of how this description is constructed for a full discussion see \cite[Chapter 3]{HolHMF} and \cite[Section 4]{HMF}. Let $A_k(\n)$ be the subspace of $A_\text{cusp}(\GL_2(F)\bs\GL_2(\A_F))$ that consists of functions such that
\begin{enumerate}
\item[(i)] $\phi(gr(\theta))=e^{-ik\theta}\phi(g)$, where $r(\theta)=\Big\{\begin{pmatrix} \cos\theta_j & -\sin\theta_j \\ \sin\theta_j & \cos\theta_j \end{pmatrix}\Big\}_j\in \text{SO}(2)^n$.
\item[(ii)] $\phi(gk_0)=\phi(g)$, where $k_0\in K_0(\n)$.
\item[(iii)] $\phi$ is an eigenfunction of the Casimir element $\Delta:=(\Delta_1,...,\Delta_n)$ as a function of $\GL_2(\R)^n$ with its eigenvalue $\lambda=\prod_{j=1}^n\frac{k_j}{2}(1-\frac{k_j}{2})$.
\end{enumerate}
Then $A_k(\n)$ is isomorphic to $S_k(\n)$. To see this we define, for each $\f=(f_1,...,f_h)\in S_k(\n)$
\begin{equation}\label{CtoA}
\phi(\gamma x_\mu g_\infty k_0)=f_\mu||_kg_\infty(\mathbf{i})
\end{equation}
where $\gamma\in\GL_2(F), g_\infty\in\GL_2^+(F_\infty), k_0\in K_0(\n)$ and $\mathbf{i}=(i,...,i)$. This is well defined because of (\ref{congsub}) and (\ref{K0}). One can also check that $\phi$ as above gives an element of $A_k(\n)$. Conversely, given $\phi\in A_k(\n)$ define \[f_\mu(z)=y^{-k/2}\phi(x_\mu g_z),\] where $g_z\mathbf{i}=z$. Then we have that $\f:=(f_1,...,f_h)\in S_k(\n)$.\\

We say that $\phi\in A_k(\n)$ is an adelic newform if it generates an irreducible cuspidal automorphic representation $\Pi=\Pi_\phi$ of $\GL_2(\AF)$ of conductor $\n$. The set of adelic newforms in $A_k(\n)$ is in bijective correspondence with the set of classical cuspidal Hilbert newforms $f\in S_k(\n)$. Precisely given $\phi$ as above it corresponds to a classical cuspidal Hilbert newform $\f=(f_1,...,f_h)$ of level $\n$ and weight $k$ defined via $f_\mu(z)=y^{-k/2}\phi(x_\mu g_z)$.\\

Given an adelic newform $\phi_\f$ attached to $\f\in S_k(\n)$, the automorphic representation of $\GL_2(\AF)$ generated by it is known to be irreducible, see for example \cite[Theorem 4.7]{HMF}. The following lemma gives an explicit relation between the action of $\Aut(\C)$ on newforms and representations. 

\begin{Lemma}\label{relation}
Let $\f=(f_1,...,f_h)$ be a cuspidal classical Hilbert newform. Given $\tau\in\Aut(\C)$ let $\f^\tau$ be defined as in \cite[Proposition 2.6]{Shi78}. Let $\Pi$ the automorphic representation generated by $\phi_\f$ and $\Pi'$ be the automorphic representation generated by $\phi_{\f^\tau}$. Then for each finite place $v$ we have \[\Pi'_v\otimes|\ |_v^{k_0/2}\cong{}^{\tau}(\Pi_v\otimes|\ |_v^{k_0/2}).\]
\end{Lemma}

\begin{proof}
See \cite[Theorem 4.19]{HMF}.
\end{proof}

\subsection{Whittaker expansion}

We now give the explicit relation between the Fourier coefficients of a classical Hilbert newform at a cusp and the Whittaker newform. We start with a general result which holds for a larger class of Hilbert automorphic forms.

\begin{Lemma}\label{general}
Let $M$ be an ideal of $\O_F$ with localisation $M_v=\varpi_v^{m_v}\O_v$ for every finite place $v$ such that for almost all $v$ we have that $m_v=0$.  Let $\phi$ be a Hilbert automorphic form in the space of $\Pi$ such that $\phi$ is right invariant by $\big(\begin{smallmatrix} 1 & \varpi_v^{m_v}\mathcal{O}_v \\0 & 1\end{smallmatrix}\big)$ for every finite place and satisfies (i) and (iii) in Section \ref{dictionary}. Let $h$ be a holomorphic function on $\mathbb{H}^n$ defined by the equation \[h(z)=j(g_z,\mathbf{i})^k\phi(g_z).\] Then $h$ has a Fourier expansion of the form 
\begin{equation}\label{FE}
h(z)=\sum_{\xi\in M^{-1}\mathfrak{D}_F^{-1}}a_h(\xi)e^{2\pi i \Tr(\xi z)}.
\end{equation}
 Moreover \[W_\phi(a(\xi)g_z)=
\begin{cases}
y^{k/2}a_h(\xi)e^{2\pi i\text{Tr}(\xi z)}, \ \text{if} \ \xi\in M^{-1}\mathfrak{D}_F^{-1}\\
0, \ \text{otherwise}.
\end{cases}\] 
\end{Lemma}

\begin{proof}
From the fact that $\phi$ satisfies (i) and (iii) in Section \ref{dictionary}, it follows that $h$ is holomorphic on $\H^n$ and $h(z+m)=h(z)$, for every $m\in M$. This shows that \[f(z)=\sum_{\xi\in F} a_h(\xi)e^{2\pi i \Tr(\xi z)},\] for some complex numbers $a_h(\xi)$.\\

Suppose $a_h(\xi)\neq0$. From the property that $h$ is invariant by $M$ we have that $\Tr(\xi m)\in\Z$ for every $m\in M$. This property is equivalent to that of $\xi\in M^{-1}\mathfrak{D}_F^{-1}$. To see this note that for every place $v<\infty$ we have that $\Tr(\xi_vm_v)\in\Z_v$, for every $m_v\in M_v$. Thus $\Tr((\varpi_v^{m_v}\xi_v)n_v)\in\Z_v$ for every $n_v\in\O_v$. This is equivalent to the following \[\varpi_v^{m_v}\xi_v\in\mathfrak{D}_v^{-1}\Longleftrightarrow \xi_v\in\varpi_v^{-m_v}\mathfrak{D}_v^{-1}\Longleftrightarrow\xi\in M^{-1}\mathfrak{D}_F^{-1}.\] This shows that $h$ has a Fourier expansion of the form (\ref{FE}). Next let $\xi\in F^\times$ and let $z=x+iy\in\H^n$. Then using the definition of $W_\phi$, (\ref{Def}), we have that 
\begin{align*} 
W_\phi(a(\xi)g_z)&=\int_{\AF/F}\phi\big(n(x')a(\xi)g_z\big)\overline{\psi(x')} \ dx'\\
&=\int_{\AF/F} \phi\big(a(\xi)n(\xi^{-1}x')g_z)\psi(-x') \ dx'\\
&=\int_{\AF/F} \phi\big(n(\xi^{-1}x')g_z)\psi(-x') \ dx', \\
&=\int_{\AF/F} \phi\big(n(x')g_z)\psi(-\xi x') \ dx', \\
&=\int_{\AF/F} \phi(n(x'+x)a(y))\psi(-\xi x') \ dx',\\
&=\int_{\AF/F} \phi(n(x')a(y))\psi(-\xi x)\psi(\xi x) \ dx',\\
&=e(\Tr(\xi x))\int_{\AF/F} \phi\big(n(x')a(y))\psi(-\xi x') \ dx'.
\end{align*}
We denote the integral in the last line by $I_\phi(y,\xi)$. We now want to split this integral into the product of two integrals, using the fundamental domain \[\AF/F=F_\infty/M\times\prod_{v<\infty}\varpi_v^{m_v}\O_v,\] using strong approximation. Therefore using $y\in F_\infty$ and the right invariance of $\phi$ by $n(\varpi_v^{m_v}\O_v)$ we have that 
\begin{align*}
\text{Vol}(F_\infty/M)I_\phi(y,\xi)&=\int_{F_\infty/M}\phi(n(x_\infty)a(y))\psi_\infty(-\xi x_\infty) \ dx_\infty\prod_{v<\infty}\int_{\O_v} \psi_v(-\xi \varpi_v^{m_v}x_v) \ dx_v,\\
&=\int_{F_\infty/M}y^{k/2}h(x_\infty+iy)e(-\Tr(\xi x_\infty)) \ dx_\infty\prod_{v<\infty}\int_{\O_v} \psi_v(-\xi\varpi_v^{m_v}x_v) \ dx_v.
\end{align*}
We have that $I_\phi(y,\xi)$ will only be non-zero when \[\int_{\O_v}\psi_v(-\xi\varpi_v^{m_v}x_v)\ dx_v\neq0\] for each $v$. Recall that $\psi_v=\psi_p\circ\Tr_{F_v/\Qp}$ and $\psi_p(x)=1$ for every $x\in\Z_p$. If the finite place integral is non-zero for each $v$ we write $\xi=\varpi_v^nu$ where $n\in\Z$ and $u\in\O_v^\times$. Therefore using (\ref{formula}) we have that \[\int_{\O_v}\psi_v(\xi\varpi_v^{m_v}x_v) \ dx_v=\int_{\O_v}\psi_v(\varpi_v^{m_v+n}uk) \ dk=
\begin{cases}
0, \ \text{if} \ m_v+n<d_v\\
1, \ \text{if} \ m_v+n\geq d_v,
\end{cases}\] 
Therefore we have that \[I_\phi(y,\xi)=
\begin{cases}
\frac{1}{\text{Vol}(F_\infty/M)}\int_{F_\infty/M}y^{k/2}h(x_\infty+iy)e(-\Tr(\xi x_\infty)) \ dx_\infty, \ \text{if} \ \xi\in M^{-1}\mathfrak{D}_F^{-1}\\
0, \ \text{otherwise}.
\end{cases}\] Using the Fourier expansion of $h$, \[h(x_\infty+iy)=\sum_{\xi'\in M^{-1}\mathfrak{D}_F^{-1}}a_h(\xi')e^{2\pi i \Tr(\xi'x_\infty)}e^{-2\pi \Tr(\xi'y)},\] we have that 
\[I_\phi(y,\xi)=
\begin{cases}
y^{k/2}a_h(\xi)e^{-2\pi i\Tr(\xi y)}, \ \text{if} \ \xi\in M^{-1}\mathfrak{D}_F^{-1}\\
0, \ \text{otherwise}.
\end{cases}\] 
Therefore for $\xi\in M^{-1}\mathfrak{D}_F^{-1}$ we have \[W_\phi(a(\xi)g_z)=e(\Tr(\xi x))y^{k/2}a_h(\xi)e(i\Tr(\xi y))=y^{k/2}a_h(\xi)e(\Tr(\xi z)),\] and for $\xi\notin M^{-1}\mathfrak{D}_F^{-1}$ we have $W_\phi(a(\xi)g_z)=0$.
\end{proof}

From now on we let $\phi$ be an adelic newform in the space of $\Pi$ and $\f$ be the corresponding classical Hilbert newform. We shall now also assume that $k_j\geq1$ for every $j=1,...,n$. Let $\sigma=\smatrix\in\Gamma_\mu(1)$ then recall that if $f_\mu$ is a Hilbert modular form for $\congsub$ we have $f_\mu||_k\sigma(z)$ is a Hilbert modular form for $\sigma^{-1}\Gamma_\mu(\n)\sigma$. We have for $\sigma\in\Gamma_\mu(1)$, using \eqref{CtoA}
\begin{align*}
f_\mu||_k\sigma(z)&=y^{-k/2}(f_\mu||\sigma g_z)(\mathbf{i})\\
&=y^{-k/2}\phi(x_\mu\iota_\infty(\sigma)g_z)\\
&=y^{-k/2}\phi(\iota_\infty(\sigma)x_\mu g_z)\\
&=y^{-k/2}\phi(\iotaf x_\mu g_z)\\
&=y^{-k/2}\phi(g_z\iotaf x_\mu),
\end{align*}
where $\iotaf x_\mu$ is an element of the finite adeles.\\

Our aim is to apply Lemma \ref{general} to the specific case of $f_\mu||_k\sigma$ where $\sigma\in\Gamma_\mu(1)$. By the above calculation, this corresponds to $(\iotaf x_\mu)\phi$ in the setup of Lemma \ref{general}. The first step to do this is by finding the ideal $M'$ of $\O_F$ corresponding to $M_v'=\varpi_v^{m'_v}\O_v$ such that $\big(\begin{smallmatrix}1 & M'_v \\ & 1\end{smallmatrix}\big)\in\sigma^{-1}\Gamma_\mu(\n)\sigma$. That is we require $\sigma\big(\begin{smallmatrix}1 & M'_v \\ & 1\end{smallmatrix}\big)\sigma^{-1}\in\congsub$. Equivalently for each place $v$, we require \[\begin{pmatrix}1 & \varpi_v^{m'_v}\O_v \\ & 1\end{pmatrix}\in\sigma^{-1}x_{\mu,v} K_v(\varpi_v^{n_v})x_{\mu,v}^{-1}\sigma.\] We can view 
\begin{equation}\label{sigma}
\sigma=x_{\mu,v}(\begin{smallmatrix}1 & \\ & \varpi_v^{d_v}\end{smallmatrix})(\begin{smallmatrix}a' & b'\\ c' & d' \end{smallmatrix}) (\begin{smallmatrix}1 & \\ & \varpi_v^{-d_v}\end{smallmatrix})x_{\mu,v}^{-1},
\end{equation}
where $A:=(\begin{smallmatrix}a' & b'\\ c' & d'\end{smallmatrix})\in\GL_2(\O_F)$ and $x_{\mu,v}=(\begin{smallmatrix}1 & \\ & \varpi_v^{t_v}\end{smallmatrix})$. Therefore 
\begin{align*}
&\begin{pmatrix} 1 & \varpi_v^{m'_v}\O_v \\ & 1\end{pmatrix} \in x_{\mu,v}\begin{pmatrix}1 & \\ & \varpi_v^{d_v}\end{pmatrix}A^{-1}\begin{pmatrix}1 & \\ & \varpi_v^{-d_v}\end{pmatrix} K_v(\varpi_v^{n_v})\begin{pmatrix}1 & \\ & \varpi_v^{d_v}\end{pmatrix}A\begin{pmatrix}1 & \\ & \varpi_v^{-d_v}\end{pmatrix}x_{\mu,v}^{-1}\\
&\Longleftrightarrow \begin{pmatrix} 1 & \varpi_v^{m'_v+t_v+d_v}\O_v \\ & 1\end{pmatrix}\in A^{-1}\begin{pmatrix}1 & \\ & \varpi_v^{-d_v}\end{pmatrix} K_v(\varpi_v^{n_v})\begin{pmatrix}1 & \\ & \varpi_v^{d_v}\end{pmatrix}A\\
&\Longleftrightarrow A\begin{pmatrix} 1 & \varpi_v^{m'_v+t_v}\O_v \\ & 1\end{pmatrix}A^{-1}\in \begin{pmatrix}1 & \\ & \varpi_v^{-d_v}\end{pmatrix}K_v(\varpi_v^{n_v})\begin{pmatrix}1 & \\ & \varpi_v^{d_v}\end{pmatrix}.\stepcounter{equation}\tag{\theequation}\label{condition}
\end{align*}
On computing the left hand side of (\ref{condition}) we require \[\begin{pmatrix} \eps^{-1}(a'd'-b'c'-a'c'\varpi_v^{m_v'+t_v+d_v}\alpha) & \eps^{-1}(a')^2\varpi_v^{m_v'+t_v+d_v}\alpha \\ -\eps^{-1}(c')^2\varpi_v^{m_v'+t_v+d_v}\alpha & \eps^{-1}(a'd'-b'c'+a'c'\varpi_v^{m_v'+t_v+d_v}\alpha) \end{pmatrix} \in y_v^{-1}K_v(\varpi_v^{n_v})y_v,\] where $\alpha\in\O_v, \ y_v=(\begin{smallmatrix}1 & \\ & \varpi_v^{d_v}\end{smallmatrix})$ and $\eps=\det A\in\O_F^\times$. If we now study the $\p$-valuation we obtain the following conditions
\begin{enumerate}
\item $2v_\p(a')+m_v'+t_v\geq 0$,
\item $2v_\p(c)+m_v'+t_v\geq n_v$.
\end{enumerate}
On rearranging these conditions we obtain \[m_v'\geq-d_v-t_v-2v_\p(a'), \ m_v'\geq -t_v-d_v+n_v-2v_\p(c').\] Hence \[m_v'\geq -t_v-d_v+\max\{-2v_\p(a'), n_v-2v_\p(c')\}.\] Let \[w_v(\sigma,\n)=\max\{-2v_\p(a'), n_v-2v_\p(c')\}=n_v-\min\{2v_\p(a')+n_v,2v_\p(c')\}.\] Then let $\mathfrak{w}_v(\sigma,\n)=\varpi_v^{w_v(\sigma,\n)}\O_v$ and $\mathfrak{w}(\sigma,\n)$ the corresponding ideal in $\O_F$, so \[\mathfrak{w}(\sigma,\n)=\n(((a')^2\n,(c')^2))^{-1}.\] This leads to $f_\mu||_k\sigma$ having a Fourier expansion of the form \[f_\mu||_k\sigma(z)=\sum_{\xi\in((t_\mu\O_F)\mathfrak{w}(\sigma,\n)^{-1})_+} a_\mu(\xi,\sigma)e(\Tr(\xi z)).\] We define 
\begin{equation}\label{FC}
c_\mu(\xi;f_\mu||_k\sigma)=N(t_\mu\O_F)^{-k_0/2}a_\mu(\xi;\sigma)\xi^{(k_0\mathbf{1}-k)/2},
\end{equation}
recall $k_0=\max\{k_1,...,k_n\}$ and $k_0\mathbf{1}=(k_0,...,k_0)$.

\begin{Remark}\label{dash}
Note from \eqref{sigma} we can view a matrix $\sigma=\smatrix\in\Gamma_\mu(1)$ as \[x_{\mu,v}(\begin{smallmatrix}1 & \\ & \varpi_v^{d_v}\end{smallmatrix})(\begin{smallmatrix}a' & b'\\ c' & d' \end{smallmatrix}) (\begin{smallmatrix}1 & \\ & \varpi_v^{-d_v}\end{smallmatrix})x_{\mu,v}^{-1},\] with $(\begin{smallmatrix}a' & b'\\ c' & d' \end{smallmatrix})\in\GL_2(\O_F)$. Therefore we have that \[\sigma_v=\begin{pmatrix} a_v & b_v \\ c_v & d_v \end{pmatrix}=\begin{pmatrix} a'_v & b'_v \\ \varpi_v^{t_v+d_v} c'_v & d_v \end{pmatrix}.\]
\end{Remark}

\begin{Prop}\label{prop}
Let $\mathfrak{w}(\sigma,\n)$ be the ideal of $\O_F$ defined by $\mathfrak{w}(\sigma,\n)=\n(\mathrm{gcd}((a')^2\n,(c')^2))^{-1}$. Let $\xi\in F^\times$. Then
\[W_{\phi}(a(\xi)g_z\iota_\emph{f}(\sigma^{-1}) x_\mu)=
\begin{cases}
y^{k/2}a_\mu(\xi;\sigma)e(\Tr(\xi z)), \ \text{if} \ \xi\in \big((t_\mu\O_F)\mathfrak{w}(\sigma,\n)^{-1}\big)_+  \\
0, \ \text{otherwise}.
\end{cases}\] 
\end{Prop}

\begin{proof}
This follows from applying Lemma \ref{general} to the automorphic form $\phi'=\Pi(\iotaf x_\mu)\phi$ and using the fact $W_{\phi}(a(\xi)g_z\iotaf x_\mu)=W_{\phi'}(a(\xi)g_z)$.
\end{proof}

We are able to evaluate $W_\infty$ via, 
\begin{equation}\label{Winfty}
W_\infty(a(\xi)g_z)=(\xi y)^{k/2}e(\Tr(\xi z)),
\end{equation}
see for example \cite[Section 3]{Sah16}. 
\section{Main Result}\label{MR}

In this section we prove our main result. We do this by applying a similar method to that used in the proof of Theorem \ref{mf2}.\\

In Section \ref{Prelim} we showed that $a_\mu(\xi;\sigma)$ can be written in terms of the global Whittaker function. Recall that the global Whittaker function can be broken up into a product of local Whittaker functions. This means that finding sufficient conditions for $\tau$ to fix $c_\mu(\xi;f_\mu||_k\sigma)$, boils down to finding sufficient conditions such that $\tau$ fixes the local Whittaker newforms. With this in mind we have the following decomposition 
\begin{equation}\label{decom}
W_\phi(g)=C\prod_{v|\infty}W_\infty(g_\infty)\prod_{v<\infty}W_v(g_v),
\end{equation}
where $C$ is a constant. This decomposition is due to the uniqueness of Whittaker functionals. We wish to compute the constant $C$. To do this let $g_\infty=g_z$ and $g_v=x_{1,v}$ for every $v<\infty$. Then using Proposition \ref{prop} and the fact that $\f$ is normalised, that is, $a_\mu(1;1)=1$ we have \[W_\phi(g_zx_1)=y^{k/2}a_1(1;1)e(\Tr(z))=y^{k/2}e(\Tr(z)).\] On the other hand using (\ref{Winfty}), (\ref{decom}) and the fact that $W_v(1)=1$ we obtain \[W_\phi(g_zx_1)=Cy^{k/2}e(\Tr(z)).\] Thus $C=1$.\\

Let $\f=(f_1,...,f_h)$ a normalised cuspidal Hilbert newform with weight $k=(k_1,...,k_n)$ such that $k_1\equiv...\equiv k_n \pmod{2}$. We let $\Q(\f)$ denote the number field generated by $c_\mu(\xi;f_\mu)$ as $\xi$ varies over $F$ and $\mu$ varies over $1\leq\mu\leq h$. Then for $\f, 1\leq\mu\leq h, \ \text{and}\ \sigma\in\Gamma_\mu(1)$ we let $\Q(\f,\mu,\sigma)$ denote the field generated by $c_\mu(\xi;f_\mu||_k\sigma)$ as $\xi$ varies over $F$. Note that $\Q(\f)$ is the compositum of the fields $\Q(f,\mu,1)$ as $\mu$ varies over $1\leq\mu\leq h$.\\

We are now able to state and prove the main result of this article.

\begin{Thm}\label{Thm}
Let $\f=(f_1,...,f_h)$ be a normalised cuspidal Hilbert newform of level $\n$ and weight $k=(k_1,...,k_n)$ with $k_1\equiv...\equiv k_n \pmod{2}$. Let $1\leq\mu\leq h$ and $\sigma=\smatrix\in\Gamma_\mu(1)$. Then $\Q(\f,\mu,\sigma)$ lies in the number field $\Q(\f)(\zeta_{N_0})$ where $N_0$ is the integer such that $N_0\Z=\n/(cdt_\mu^{-1}\mathfrak{D}_F^{-1},\n)\cap\Z$.
\end{Thm}

\begin{proof}
Let $\tau\in\Aut(\C)$ fix $\Q(\f)(\zeta_{N_0})$ where $N_0$ is the unique positive integer such that $N_0\Z=\n/(cdt_\mu^{-1}\mathfrak{D}_F^{-1},\n)\cap\Z$. Thus specifically $\tau$ fixes $\Q(\f)$ and all the $\zeta_{N_0}$ roots of unity.\\

From Proposition \ref{prop} and evaluating $W_\infty$ we have that 
\begin{align*}
a_\mu(\xi;\sigma)&=y^{-k/2}e(-\Tr(\xi z))W_{\phi}(a(\xi)g_z\iotaf x_\mu)\\
&=y^{-k/2}e(-\Tr(\xi z))(\xi y)^{k/2}e(\Tr(\xi z))\prod_{v<\infty}W_v(a(\xi)\iotaf x_\mu)\\
&=\xi^{k/2}\prod_{v<\infty}W_v(a(\xi)\iotaf x_\mu).
\end{align*}
Recall from (\ref{FC}) we have \[c_\mu(\xi;f_\mu||_k\sigma)=N(t_\mu\O_F)^{-k_0/2}\xi^{(k_0\mathbf{1}-k)/2}a_\mu(\xi;\sigma).\] Therefore 
\begin{equation}\label{cmu}
c_\mu(\xi;f_\mu||_k\sigma)=N(t_\mu\O_F)^{-k_0/2}\xi^{k_0\mathbf{1}/2}\prod_{v<\infty}W_v(a(\xi)\iotaf x_\mu).
\end{equation}
From the fact that $\tau$ fixes $\Q(\f)$ and using \cite[Proposition 2.6]{Shi78}, multiplicity one and Lemma \ref{relation} we have that $\f=\f^\tau$. Therefore, using Lemma \ref{relation} we have \[\Pi_v\otimes| \ |_v^{k_0/2}\cong{}^{\tau}(\Pi_v\otimes| \ |_v^{k_0/2}).\] So, by (\ref{action}), we have that for each non-archimedean place $v$
\begin{equation}\label{tauWv}
\tau(W_v(a(\xi)\iotaf x_\mu))\tau\big((|\xi|_v|t_\mu|_v)^{k_0/2}\big)=W_v(a(\alpha_\tau)a(\xi)\iotaf x_\mu)(|\xi|_v|t_\mu|_v)^{k_0/2}.
\end{equation}

Since all the $N_0$ roots of unity are fixed by $\tau$ we have that \[\alpha_\tau\equiv1 \pmod{\n_v/(c_vd_v\varpi_v^{-t_v-d_v},\n_v)},\] where $c_v$ and $d_v$ are the $v$ parts of $c$ and $d$ respectively. Consider the product \[x_\mu^{-1} \begin{pmatrix} a & b \\ c & d \end{pmatrix} \begin{pmatrix} \alpha_\tau & \\ & 1 \end{pmatrix} \begin{pmatrix} a & b \\ c & d \end{pmatrix}^{-1}x_\mu.\] On using \eqref{sigma} we have this product is equal to 
\begin{equation}\label{product}
\begin{pmatrix} 1 & \\ & \varpi_v^{d_v}\end{pmatrix}\begin{pmatrix} a' & b' \\ c' & d' \end{pmatrix}\begin{pmatrix} 1 & \\ & \varpi_v^{-d_v} \end{pmatrix}\begin{pmatrix} \alpha_\tau & \\ & 1 \end{pmatrix} \begin{pmatrix} 1 & \\ & \varpi_v^{d_v} \end{pmatrix} \begin{pmatrix} a' & b' \\ c' & d' \end{pmatrix} \begin{pmatrix} 1 & \\ & \varpi_v^{-d_v} \end{pmatrix}.
\end{equation}
On evaluating \eqref{product} we have \[\begin{pmatrix}\eps^{-1}(a'd'\alpha_\tau-b'c') & \eps^{-1}a'b'(1-\alpha_\tau)\varpi_v^{-d_v}\\ \eps^{-1}c'd'(\alpha_\tau-1)\varpi_v^{d_v} & \eps^{-1}(a'd'-b'c'\alpha_\tau)\end{pmatrix}, \ \text{where} \ \eps=\det(\begin{smallmatrix}a' & b' \\ c' & d' \end{smallmatrix}).\] We want to show that this matrix is an element of $K_v(\n)$. Therefore we need the following conditions to be satisfied 
\begin{enumerate}
\item $\eps^{-1}a'b'(1-\alpha_\tau)\varpi_v^{-dv}\in\mathfrak{D}_v^{-1}$
\item $\eps^{-1}c'd'(\alpha_\tau-1)\varpi_v^{d_v}\in \n_v\mathfrak{D}_v$
\item $\eps^{-1}(a'd'-b'c'\alpha_\tau)\in\O_v$.
\end{enumerate}
Recall that $a',b',c'$ and $d'$ are elements of $\O_F$ and $\eps^{-1}\in\O_F^\times$ so conditions (1) and (3) are satisfied with no additional assumptions on $\alpha_\tau$. Now consider condition (2). From the fact we have $\alpha_\tau\equiv1 \pmod{\n_v/(c_vd_v\varpi_v^{-t_v-d_v},\n_v)}$ which is equivalent to $\alpha_\tau\equiv1 \pmod{n_v/(c'_vd'_v,\n_v)}$, using Remark \ref{dash}. Thus we have that condition (2) is also satisfied. Therefore \[\begin{pmatrix}\eps^{-1}(a'd'\alpha_\tau-b'c') & \eps^{-1}a'b'(1-\alpha_\tau)\varpi_v^{-dv}\\ \eps^{-1}c'd'(\alpha_\tau-1)\varpi_v^{d_v} & \eps^{-1}(a'd'-b'c'\alpha_\tau)\end{pmatrix}\in K_v(\n).\] Recall that in general for $W_v(g_1)=W_v(g_2)$ it suffices that $g_1^{-1}g_2\in K_v(\n)$. Therefore we have that 
\begin{equation}\label{Wvfin}
W_v(a(\xi)\iotaf x_\mu)=W_v(a(\alpha_\tau)a(\xi)\iotaf x_\mu),
\end{equation}
for every finite place. Combining (\ref{tauWv}) and (\ref{Wvfin}), we have 
\begin{equation}\label{fraction}
\frac{\tau\big(W_v(a(\xi)\iotaf x_\mu)\big)}{W_v(a(\xi)\iotaf x_\mu)}=\frac{|\xi t_\mu|_v^{k_0/2}}{\tau(|\xi t_\mu|_v)^{k_0/2}}.
\end{equation}
We now consider the quotient \[\frac{\tau\big(c_\mu(\xi;f_\mu||_k\sigma)\big)}{c_\mu(\xi;f_\mu||_k\sigma)}.\] To prove the result it now suffices to show this quotient is equal to one. Therefore by (\ref{cmu}) and (\ref{fraction}) we have 
\begin{align*}
\frac{\tau\big(c_\mu(\xi;f_\mu||_k\sigma)\big)}{c_\mu(\xi;f_\mu||_k\sigma)}&=\frac{\tau\big(N(t_\mu\O_F)^{-k_0/2}\big)\tau(\xi^{k_0/2})}{N(t_\mu\O_F)^{-k_0/2}\xi^{k_0/2}}\prod_{v<\infty}\frac{|\xi t_\mu|_v^{k_0/2}}{\tau(|\xi t_\mu|_v)^{k_0/2}}\\
&=\frac{\tau\big(\prod_{v<\infty}|t_\mu|_v^{k_0/2}\big)}{\prod_{v<\infty}|t_\mu|_v^{k_0/2}}\frac{\tau\big(\prod_{v|\infty}|\xi|_v^{k_0/2}\big)}{\prod_{v|\infty}|\xi|_v^{k_0/2}}\frac{\big(\prod_{v<\infty}|t_\mu|_v|\xi|_v\big)^{k_0/2}}{\tau\big(\prod_{v<\infty}|t_\mu|_v|\xi|_v\big)^{k_0/2}}\\
&=\frac{\tau\big(\prod_{v|\infty}|\xi|_v^{k_0}\big)}{\prod_{v|\infty}|\xi|_v^{k_0}}, \ \text{using global norm formula}.
\end{align*}
We have that $\prod_{v|\infty}|\xi|_v=|N(\xi)|\in\Q$. Hence \[\frac{\tau\big(\prod_{v|\infty}|\xi|_v^{k_0}\big)}{\prod_{v|\infty}|\xi|_v^{k_0}}=\frac{\prod_{v|\infty}|\xi|_v^{k_0}}{\prod_{v|\infty}|\xi|_v^{k_0}}=1.\] Therefore we have that $\tau$ fixes $c_\mu(\xi;f_\mu||_k\sigma)$ and so \[c_\mu(\xi;f_\mu||_k\sigma)\in\Q(\f)(\zeta_{N_0}).\]
\end{proof}

\begin{Remark}\label{special}
We see that in the special case of $F=\Q$ we obtain the result of Brunault and Neururer for newforms.
\end{Remark}

\begin{Question}
Our method to prove Theorem \ref{mf2} and Theorem \ref{Thm} was to find sufficient conditions not necessary ones. Therefore one can ask if the number field $\Q(\f)(\zeta_{N_0})$ is optimal, that is, do we have $\Q(\f,\mu,\sigma)=\Q(\f)(\zeta_{N_0})$?\\

Note that if $F=\Q$ this is known to be true \cite[Theorem 7.6]{FEatC}, so one can restrict their attention to the case of $n>1$.
\end{Question}

\begin{Question}
It would be interesting to generalise the results in this paper to the setting of Hilbert newforms with non-trivial central character and specifically what will the number field be.
\end{Question}

\begin{tiny}
SCHOOL OF MATHEMATICAL SCIENCES, QUEEN MARY UNIVERSITY OF LONDON, LONDON E1 4NS, UK

\textit{E-mail address}: tim.davis@qmul.ac.uk 

\end{tiny}

\end{document}